\documentclass[a4paper,12pt]{article}
\usepackage[english]{babel}
\usepackage{hyperref}
\usepackage{indentfirst}
\usepackage{amsfonts,amssymb,amsthm,amsmath}
\theoremstyle{definition}
\newtheorem{De}{Definition}
\theoremstyle{theorem}
\newtheorem{Th}{Theorem}
\theoremstyle{lemma}
\newtheorem{Lm}{Lemma}
\theoremstyle{plain}
\newtheorem{Pl}{Proposition}
\theoremstyle{remark}

\theoremstyle{plain}

\newcounter{tmp}

\newcommand{\ir}{[0,T]\times \mathbb{R}^{n}}

\newcommand{\itp}[1]{\int_{[0,#1]\times P}}
\newcommand{\ittp}[2]{\int_{[#1,#2]\times P}}
\newcommand{\iou}{\int_{\Omega\times\mathcal{U}}}
\begin{document}
\title{Minimax Approach to First-Order Mean Field Games}
\author{Yurii Averboukh\footnote{Krasovskii Institute of Mathematics and Mechanics UrB RAS, S.~Kovalevskaya str. 16, Ekaterinburg, 620990, Russia, e-mail: ayv@imm.uran.ru, averboukh@gmail.com.}}
\date{}
\maketitle
\begin{abstract}
The paper is devoted to the first-order mean field game system in the case when the distribution of players can contain atoms. The proposed definition of a generalized solution is    based on the minimax approach to the Hamilton-Jacobi equation. We prove the existence of the generalized (minimax) solution of the mean filed game system using the Nash equilibrium in the auxiliary differential game with infinitely many identical players. We show that the minimax solution of the original system provide the $\varepsilon$-Nash equilibrium in the  differential game with finite number of players.
\end{abstract}
{\small \textbf{Keywords:} {Mean-field-games, Hamilton--Jacobi equations, minimax solution, Nash equilibrium, differential game with infinitely many players.}

\vspace{6pt}\noindent\textbf{AMS 2010 Subject Classification.} 35F31, 35F21, 49N70, 91A23, 91A10.}
\section{Introduction}
%The paper is devoted to the first-order mean field game (MFG) system of the general form.

The mean field game approach was proposed independently by J.-M. Lasry and P.-L. Lions (see \cite{Lions_Guent}, \cite{Lions01}--\cite{Lions}) and by M. Huang, R.P. Malham\'{e} and P. Caines (see \cite{Huang1}--\cite{Huang5}). It is used to describe the control process with the $N$ identical weakly coupled participants  by studying the limit case  $N\rightarrow\infty$. In the limit case the dynamics and outcome of each player depend on the state of the player, his control and the distribution of the players on the Euclidian space, while the distribution of players is determined by the dynamics. This leads to the MFG system consisting of two PDEs: Hamilton-Jacobi equation and kinetic/Kolmogorov equation. Hamilton--Jacobi equation describes the evolution of value function; the solution of kinetic equation determines the distribution of players' state.

Primary mean field games are studied for the second--order case. This case corresponds to stochastic control processes. Note that for stochastic case the mean field game theory is developed for the  nonlinear Markov processes of the general form~\cite{Kolokoltsev},~\cite{Kolokoltstov13}. For the information on recent progress of mean field games we refer to the survey \cite{Gomes} and references therein.

First-order mean field games were studied for the case when the distribution of players' states is absolutely continuous with respect to the Lebesgue measure (see \cite{Bagagiolo}, \cite{Cardal_mfg}, \cite{cardal_lions}, \cite{Lions}, \cite{lions_frans}). The existence and uniqueness of the solution to the MFG system is  proved for the case when the dependence on the density of the players' state distribution is of nonlocal nature and smoothing \cite{Lions}, or for the case when dependence on the density of the players' state distribution is local and Hamiltonian has a superlinear growth in the gradient variable \cite{Cardal_mfg}. The key idea of those works is that the MFG system
is understood as an optimality condition for certain optimization problem. For some cases the existence can be established using fixed point arguments applied directly to MFG system \cite{Bagagiolo}, \cite{lions_frans} (see also \cite{cardal_lions}).

Another approach is based on a random variables point of view. It was implemented to extended mean field games  \cite{Gomes_extended}. In that paper the second (kinetic equation) is replaced with the system of ODEs. This requires some smoothness of the Hamiltonian, moreover the coercivity condition is imposed on the conjugate of the Hamiltonian (see~\cite{Gomes},~\cite{Gomes_extended}).

In this paper we introduce the notion of a minimax solution of the MFG system and prove its existence for the case when the distribution of players' states can contains  atoms. The main assumption is that the Hamiltonian is continuous and Lipschitz continuous with respect to gradient variable. %We assume that the Hamiltonian is continuous and % and  Lipschitz continuous with respect to space, measure and gradient variables.
The  minimax solutions were proposed for Hamilton--Jacobi equations by Subbotin \cite{Subb_book}. The concept of minimax solution goes back to the notions of stability proposed by Krasovskii and Subbotin for zero-sum differential  games \cite{NN_PDG_en}. The function of position is a minimax solution if its epigraph and hypograph are viable under certain differential inclusions. The definition of the minimax solution to Hamilton--Jacobi equation can be rewritten in the infinitesimal form \cite{Subb_book}. Moreover, the minimax solutions are equivalent to the viscosity solutions.

The definition of the minimax solution to the MFG system proposed in this paper means that the Hamilton--Jacobi equation holds in minimax sense, when the kinetic equation is replaced by the following condition: the distribution of  players' states is determined by the measure on the set of trajectories viable in the graph of the solution of Hamilton--Jacobi equation.
We prove that the minimax solution coincides with the classical solution if it exists. Our approach is close to the approach based  random variables formulation \cite{Gomes_extended}. As in that paper the distribution of players is determined by the measure on trajectories. However, the minimax approach works in the nonsmooth case.

The key idea of the paper is to consider the game with infinitely many identical players corresponding to the original MFG system. We prove the existence of Nash equilibrium for such game.  Further, we consider the function equal to the optimal outcome of the sample player placed at the given position. This function and the distribution of  players' states given by Nash equilibrium form the minimax solution of the MFG system. Additionally, we construct the $\varepsilon$-Nash equilibrium in the game of finite number of players. Here we assume that the players use open--loop random strategies.

The paper is organized as follows. We start in Section \ref{sec_MFG} with the statement of the problem. Then, in Section \ref{sec_minimax} we give the notion of minimax solution to the MFG system. In  Section \ref{sec_properties} we derive the necessary conditions for minimax  solutions in the infinitesimal form; in addition, we study the relation between minimax and classical solutions for MFG systems. Section \ref{sec_infinitely_many} includes the existence result for the Nash equilibrium in the auxiliary differential game with infinitely many players. In Section \ref{sec_dynamic_programming} this result is used to prove  the existence of the minimax solution to the original MFG system. Finally, Section \ref{sec_approximate} presents the construction of the approximate Nash equilibrium for the finitely many players differential game.

\section{Setting the Problem}\label{sec_MFG}
\subsection{Mean Field Game System}
We consider the  first-order mean field game system
%\begin{subequations}
%\left\{
\begin{alignat}{1}
     &\frac{\partial V}{\partial t}+H(t,x,\mu[t],\nabla V)=0,\label{HJB}\ \ V(T,\cdot)=\sigma(\cdot) \\
    &{}\frac{d}{dt}\mu[t]=\mu[t]\left\langle\frac{\partial H}{\partial p}(t,x,\mu[t],\nabla V),\nabla\right\rangle, \ \ \mu[0]=m_0(\cdot)\label{meanfield}.
    \end{alignat}
%\right.
%\end{subequations}
Here $t\in [0,T]$, $x\in\mathbb{R}^n$, $V$ is a function from $\ir$ to $\mathbb{R}$, $\mu[\cdot]$ is a measure-valued function of $t$, i.e. for all $t$ $\mu[t]$ is a probabilistic measure on $\mathbb{R}^n$, $\partial H/\partial p$ denotes the derivative of the Hamiltonian $H$ with respect to 4th variable.  We don't assume that $\mu[t]$ is absolute continuous.

Below we consider the case when $V$ isn't smooth and it satisfies equation (\ref{HJB}) in minimax (viscosity) sense. The kinetic equation is replaced with the condition which doesn't require the existence of $\partial H/\partial p$.

The kinetic  equation (\ref{meanfield}) is written in the operator form.  In this form the MFG systems were studied in \cite{Kolokoltsev}, \cite{Kolokoltstov13} for stochastic case under some smoothness conditions on coefficients. Equation (\ref{meanfield}) is the kinetic equation used in \cite{Kolokoltsev}, \cite{Kolokoltstov13} in the case when  the terms corresponding to jumps and Brownian motion are equal to 0.

Note that if $\mu[t]$ is a absolute continuous and $q(t,\cdot)$ is its density, then the  kinetic equation (\ref{meanfield}) takes the form
\begin{equation}\label{lions_meanfield}
\frac{\partial q}{\partial t}+{\rm div}\left(q \frac{\partial H(t,x,\mu,\nabla V)}{\partial p}\right)=0,\ \ q(0,\cdot)=q_0(\cdot).
\end{equation}
  Here $q_0$ is a density of the measure $m_0$. %The equation (\ref{lions_meanfield}) were considered in \cite{Lions}, \cite{Cardal_mfg}.

\subsection{Notions and Assumptions}\label{sec_notions}
We consider an element of  $\mathbb{R}^{n+1}$ as a pair $(x,z)$, $x\in \mathbb{R}^n$, $z\in\mathbb{R}$. Put $$\mathrm{proj}_1(x,z)\triangleq x.$$

If $S$ is a Banach space, then $C(S)$ is the space of   continuous functions from $S$ to $\mathbb{R}$, $C_b(S)$ is the space of bounded function $S\rightarrow\mathbb{R}$. The spaces $C(S)$ and $C_b(S)$ are equipped with the uniform norm $\|\varphi\|=\sup_{s\in S}|\varphi(s)|$. Further, $C^k(S)$ is the space of $k$-times continuously differentiable functions $S\rightarrow\mathbb{R}$, $C_b^k(S)$ is the space of $k$-times continuously differentiable functions $S\rightarrow\mathbb{R}$ such that any its partial derivative is bounded.

Denote by $C([0,T],\mathbb{R}^{n+1})$ the space of all continuous functions $(x(\cdot),z(\cdot)):[0,T]\rightarrow \mathbb{R}^{n+1}$. The norm on $C([0,T],\mathbb{R}^{n+1})$ is given by $$\|(x(\cdot),z(\cdot))\|=\sup_{t\in [0,T]}(\|x(t)\|+|z(t)|). $$  If $(x(\cdot),z(\cdot))\in C([0,T],\mathbb{R}^{n+1})$, then put $$e_t(x(\cdot),z(\cdot))=x(t). $$

We denote the Borel $\sigma$-algebra on $S$ by $\mathcal{B}(S)$. If $m$ is a Borel measure on $S$, then  ${\rm supp}(m)$ denotes the support of $m$.
If $R$ is a Banach space and $h$ is a Borel map from $S$ to $R$, then $h\#m$ denotes an image measure of $m$ by $h$:
$$\int_R\varphi(r)(h\#m)(dr)=\int_S\varphi(h(s))m(ds) \ \ \forall\phi\in C_b(R). $$

The set of all Borel probability measures on $S$ is denoted by $\mathcal{P}(S)$.
We endow the set $\mathcal{P}(\mathbb{R}^n)$ with the weak topology. This topology is generated by the Kantorovich--Rubinstein distance \cite{Villani}:
\begin{multline*}W(m',m'')=\sup\left\{\int_{\mathbb{R}^n}\phi d(m'-m''):\phi\in {\rm Lip}_1\right\}\\=\inf_{\pi\in\Pi(m',m'')}\int_{\mathbb{R}^n\times\mathbb{R}^n}\|x'-x''\|\pi(dx',dx''). \end{multline*} Here $${\rm Lip}_1=\{\phi:|\phi(x')-\phi(x'')|\leq \|x'-x''\|\},$$
$\Pi(m',m'')$ denotes the set of probabilistic measures $\pi$ on $\mathbb{R}^n\times\mathbb{R}^n$ such that $\pi( Y\times\mathbb{R}^n)=m'( Y)$, $\pi(\mathbb{R}^n\times  Y)=m''( Y)$ for all measurable $A\subset \mathbb{R}^n$.

Let $\mathcal{M}$ be the set of all continuous functions $\mu:[0,T]\rightarrow \mathcal{P}(\mathbb{R}^n)$. If $\mu\in\mathcal{M}$, then $\mu[t]$ is a distribution of players at time $t$. If $\mu,\nu\in\mathcal{M}$, then define   $\mathcal{W}(\mu,\nu)$ by the rule $$\mathcal{W}(\mu,\nu)\triangleq\sup_{t\in [0,T]}W(\mu[t],\nu[t]). $$ Note that $\mathcal{W}$ is a distance, and $\mathcal{M}$ is a Banach space. The measure-valued function $\mu\in\mathcal{M}$ can be considered as an external field for the Hamilton--Jacobi equation (\ref{HJB}).

We assume that
$$H(t,x,m,p)=\max_{u\in P}[\langle p,f(t,x,m,u)\rangle-g(t,x,m,u)]. $$ Here $m$ is a probabilistic measure on $\mathbb{R}^n$, the variable $p$ denotes $\nabla V$.

%The field $\mu$ can be considered as a function $t\mapsto\mu[t]\in\mathcal{P}(\mathbb{R}^n)$.

We assume that the following conditions hold true.
\renewcommand{\theenumi}{\arabic{enumi}}
\begin{enumerate}
  \item $P$ is compact;
  \item $f$ and $g$ are continuous;
  \item $f$ is Lipschitz continuous with respect $x$ and $m$, i.e. there exist constants $L_{f,x}$ and $L_{f,m}$ such that
  $$\|f(t,x',m,u)-f(t,x'',m,u)\|\leq L_{f,x}\|x'-x''\|,$$  $$\|f(t,x,m',u)-f(t,x,m'',u)\|\leq L_{f,m} W(m',m'');$$
  \item $\sigma$ is continuous;
  \item the support of the measure $m_0$ is a compact set $G_0\subset\mathbb{R}^n$.
\end{enumerate}

There exists a compact set $G\subset\mathbb{R}^n$ such that for each external field $\mu$, control $u$ and the function  $x(\cdot)$ satisfying the condition $$ \dot{x}=f(t,x,\mu[t],u(t)), \ \ x(0)\in G_0 $$  the inclusion $x(t)\in G $ holds true for $t\in [0,T]$.  Put
 \begin{equation}\label{K_def}
 K\triangleq\sup\{\|f(t,x,m,u)\|:t\in [0,T],x\in G,m\in\mathcal{P}(E),u\in P\}.
\end{equation}
Since $G$, $P$ and $\mathcal{P}(G)$ are compact, $K$ is finite.

\section{Minimax Solution}\label{sec_minimax}

The solution  of   equation (\ref{HJB}) can be  nonsmooth even if $H$ is smooth \cite{Subb_book}. In this case  equation (\ref{HJB}) should be satisfied in the minimax (viscosity) sense. There exist several (equivalent) definitions of minimax solution \cite{Subb_book}.
We will use the definition involving the notion of viability.

Let multivalued maps $(t,x,m)\mapsto E^-(t,x,m)\subset \mathbb{R}^{n+1}$ and $(t,x,m,a)\mapsto E^+(t,x,m,a)\subset \mathbb{R}^{n+1}$, $a\in \mathcal{A}$ satisfy the following conditions
\begin{list}{\rm (E\arabic{tmp})}{\usecounter{tmp}}
  \item the sets $E^-(t,x,m)$ and $E^+(t,x,m,a)$ are nonempty, closed and convex;
  \item the multivalued function $(t,x,m)\mapsto E^-(t,x,m)$ is upper semicontinuous; for all $a$ the mapping $(t,x,m)\mapsto E^+(t,x,m,a)$ is upper semicontinuous.
  \item $H(t,x,m,p)=\max\{\langle p,\xi\rangle+\zeta:(\xi,\zeta)\in E^-(t,x,m)\};$
  \item for any $t,x,m$ and $p$ there exists $a_*\in \mathcal{A}$ such that
  \begin{multline*}H(t,x,m,p)=\min\{\langle p,\xi\rangle+\zeta:(\xi,\zeta)\in E^+(t,x,m,a_*)\}\\ \geq \min\{\langle p,\xi\rangle+\zeta:(\xi,\zeta)\in E^+(t,x,m,a)\}, \ \ \forall a\in \mathcal{A}.
                     \end{multline*}
\end{list}

One can define  maps $E^-$ and $E^+$ by the rule  $$E^-(t,x,m)={\rm co}\{(f(t,x,m,u),g(t,x,m,u)):u\in P\},$$ $$ E^+(t,x,m,u)=\{(f(t,x,m,u),g(t,x,m,u))\}.$$ In this case the index set $\mathcal{A}$ is equal to $P$.

The function $V$ is a minimax solution for the given external field $\mu\in\mathcal{M}$ (see \cite[Definitions M3, U2, and L2]{Subb_book}) iff
\renewcommand{\theenumi}{\roman{enumi}}
\begin{enumerate}
\item $V(T,\cdot)=\sigma(\cdot,\mu[T])$;
\item
for any $a\in \mathcal{A}$ the epigraph of $V$ is viable under the differential inclusion
$$%\begin{equation}\label{upper_in}
    (\dot{x},\dot{z})\in E^+(t,x(t),\mu[t],a);
$$%\end{equation}
\item the hypograph of $V$ is viable under differential inclusion
\begin{equation}\label{lower_in}
(\dot{x},\dot{z})\in E^{-}(t,x,\mu[t]).
\end{equation}
\end{enumerate}

Note that the definition of a minimax solution doesn't depend on the choice of the maps $E^-$ and $E^+$ \cite{Subb_book}.

Both conditions can be rewritten in the infinitesimal form \cite[Theorem 6.4]{Subb_book}. First, let us introduce the upper Hadamard derivative $d^+$ and lower Hadamard derivative~$d^-$.
$$d^-V(t,x,\alpha,\xi)=\liminf_{\delta\downarrow 0, \alpha'\rightarrow\alpha, \xi'\rightarrow \xi}\frac{V(t+\delta\alpha',x+\delta \xi')-V(t,x)}{\delta}, $$
$$d^+V(t,x,\alpha,\xi)=\limsup_{\delta\downarrow 0, \alpha'\rightarrow\alpha, \xi'\rightarrow \xi}\frac{V(t+\delta\alpha',x+\delta \xi')-V(t,x)}{\delta}. $$

The equivalent definition of the minimax solution is the following. The function $V$ is a minimax solution of  equation (\ref{HJB}) iff
\begin{enumerate}
\item $V(T,\cdot)=\sigma(\cdot,\mu[T])$;
\item $
    \inf\{d^+V(t,x,1,\xi)-\zeta:(\xi,\zeta)\in E^-(t,x,\mu[t])\}\geq 0;
$
\item for each $a\in \mathcal{A}$
$\sup\{d^-V(t,x,1,\xi)-\zeta:(\xi,\zeta)\in E^+(t,x,\mu[t],a)\}\leq 0 .$
\end{enumerate}
Note that $V$ is a minimax solution of equation (\ref{HJB}) if and only if for any $(t_0,x_0)\in\ir$ $V(t_0,x_0)$ is a value  of the control problem
$$\mbox{maximize }J(x(\cdot),u(\cdot),\mu)=\sigma(x(T),\mu[T])-\int_{t^0}^Tg(t,x(t),\mu[t],u(t))dt $$
subject to
$$\dot{x}(t)=f(t,x(t),\mu[t],u(t)), \ \ x(t_0)=x_0. $$

For a function of position $V$ and an external field $\mu$ put
$$\mathcal{S}[V,\mu]\triangleq\{(x(\cdot),z(\cdot)):(\dot{x}(t),\dot{z}(t))\in E^-(t,x(t),\mu[t]), \ \ z(t)=V(t,x(t))\}. $$ The set $\mathcal{S}[V,\mu]$ is a set of solutions of the inclusion $(\dot{x},\dot{z})\in E^-(t,x,\mu[t])$ viable  in ${\rm gr}V$. %Note that $\mathcal{S}[V,\mu]\subset C([0,T],\mathbb{R}^{n+1})$.

\begin{De}\label{def_solution}
We say that $({V},{\mu})\in C(\ir)\times\mathcal{M}$ is a minimax solution of  system (\ref{HJB}), (\ref{meanfield}) iff
\renewcommand{\theenumi}{\arabic{enumi}}
\begin{enumerate}
  \item ${V}$ is a minimax solution of  equation (\ref{HJB});
  \item ${\mu}[0]=m_0 $
  \item there exists a measure $\chi$ on $\mathcal{S}[V,\mu]$ such that $\mu[t]=e_t\#\chi$, i.e. for any $\phi\in C_1[0,T]$ the following equality holds true: $$\int_{\mathbb{R}^n}\phi(x)\mu[t](dx)=\int_{\mathcal{S}[V,\mu]}\phi(x(t))\chi(d(x(\cdot),z(\cdot))).$$
 \end{enumerate}
\end{De}

Note that the definition doesn't depend on the choice of the maps $E^-$ and $E^+$. %Indeed, let $E^-(t,x,m)$  satisfy conditions (E1)--(E3).
Indeed, denote $$H^*(t,x,m,\xi)=\sup_{p\in\mathbb{R}^n}[\langle\xi,p\rangle-H(t,x,m,p)]. $$ We have that
\begin{multline*}\{\xi\in\mathbb{R}^n:\mbox{ there exists }\zeta\mbox{ such that }(\xi,\zeta)\in E^-(t,x,m)\}=\mathcal{H}(t,x,m)\\=\left\{\xi\in\mathbb{R}^n:H^*(t,x,m,\xi)<\infty\right\}. \end{multline*} Moreover, if $(\xi,\zeta)\in E^-(t,x,m)$, then $\zeta\geq H^*(t,x,m,\xi)$, and the pair $(\xi,H^*(t,x,m,\xi))$ is an element of $E^-(t,x,m)$.
Therefore, the set $\mathcal{S}[V,\mu]$ consists of the solutions for the inclusion
$$(\dot{x}(t),\dot{z}(t))\in\{(\xi,\zeta):\xi\in \mathcal{F}(t,x,\mu[t]), \ \  \zeta=H^*(t,x,m,\xi)\}. $$It is determined only by the Hamiltonian $H$.

In Section \ref{sec_dynamic_programming} the existence theorem for the minimax solution of system (\ref{HJB}),~(\ref{meanfield}) is proved.

\section{Properties of Minimax Solution}\label{sec_properties}
For brevity, denote by $\mathcal{T}V(t,x)$  the tangent cone at $(t,x)$ to ${\rm gr}V$:
$$\mathcal{T} V(t,x)=\left\{(v,s):\liminf_{\delta\downarrow 0}\left|\frac{V(t+\delta,x+\delta v)}{\delta}-s\right|=0\right\}. $$

\begin{Pl}\label{pl_infinitisimal}
  If $(V,\mu)$ is a minimax solution to system (\ref{HJB}), (\ref{meanfield}), then there exists a measurable function $b(t,x)$ such that $b(t,x)\in  \mathrm{proj}_1(\overline{\mathrm{co}}\mathcal{T} V(t,x,{\mu})\cap E^-(t,x,\mu[t]))$ and
      \begin{equation}\label{integral_transform}
      \frac{d}{dt}\int_{\mathbb{R}^n}\phi(x){\mu}[t](dx)=\int_{\mathbb{R}^n}\langle b(t,x),\nabla\phi(x)\rangle{\mu}[t](dx)\ \ \forall \phi\in C^1_0(\mathbb{R}^n).
      \end{equation}
\end{Pl}
\begin{proof}
We may assume that for any $(x(\cdot),z(\cdot))\in \mathcal{S}[V,\mu]$ and  $t\in [0,T]$ $(\dot{x}(t),\dot{z}(t))\in \mathcal{T} V(t,x(t))$. Indeed,
since the motion from $\mathcal{S}[V,\mu]$ is viable in $\mathrm{gr} V$ we have that $(\dot{x}(t),\dot{z}(t))\in \mathcal{T} V(t,x) $ for a.e. $t\in [0,T]$.

If $\phi\in C_0^1(\mathbb{R}^n)$, $(x(\cdot),z(\cdot))\in \mathcal{S}[V,\mu]$, then
$$\frac{d}{d t}\phi(x(t))=\langle \nabla\phi(x(t) ),\dot{x}(t)\rangle\ \ \mbox{ a.e. }t\in [0,T]. $$ Therefore, for any  $\psi\in C^1([0,T])$ such that $\psi(0)=\psi(T)=0$ we have that
$$\int_0^T\psi(t)\frac{d}{d t}\phi(x(t))dt=\int_0^T\psi(t)\langle\nabla \phi(x(t)),\dot{x}(t)\rangle dt. $$ Further,
\begin{multline*}-\int_{\mathcal{S}[V,\mu]}\int_0^T\psi'(t)\phi(x(t))dt\chi(d(x(\cdot),z(\cdot)))\\=\int_{\mathcal{S}[V,\mu]}\int_0^T\psi(t)\langle\nabla \phi(x(t)),\dot{x}(t)\rangle dt\chi(d(x(\cdot),z(\cdot))). \end{multline*}
Using Fubini's Theorem we get
\begin{multline}\label{chi_fubb}
-\int_0^T\int_{\mathcal{S}[V,\mu]}\psi'(t)\phi(x(t))\chi(d(x(\cdot),z(\cdot)))dt\\= \int_{\mathcal{S}[V,\mu]}\int_0^T\psi(t)\langle\nabla \phi(x(t)),\dot{x}(t)\rangle dt\chi(d(x(\cdot),z(\cdot))).%\int_0^T\int_{{\rm Sol}(0)}\psi(t)\langle\nabla \phi(x(t)),\dot{x}(t)\rangle \chi(d(x(\cdot)))dt.
\end{multline}

Let $$\mathcal{S}[V,t_*,x_*,\mu]=\{(x(\cdot),z(\cdot))\in\mathcal{S}[V,\mu]:x(t_*)=x_*\}.$$ It follows from \cite[Theorem 10.4.6]{Bogachev} that there exists a system of measures $\chi_{t_*,x_*}$ on $\mathcal{S}[V,t_*,x_*,\mu]$ such that for all $\varphi\in C_b(\mathcal{S}[V,\mu])$
\begin{multline*}\int_{\mathcal{S}[V,\mu]}\varphi(x(\cdot),z(\cdot))\chi(d(x(\cdot),z(\cdot)))\\= \int_{0}^Tdt_*\int_{\mathbb{R}^n}\mu[t_*](dx_*)\int_{\mathcal{S}[V,t_*,x_*,\mu]}\varphi(x(\cdot),z(\cdot))\chi_{t_*,x_*}(d(x(\cdot),z(\cdot))). \end{multline*}

Applying this formula to (\ref{chi_fubb})  we get the equality
\begin{multline*}-\int_0^T\psi'(t)\int_{\mathbb{R}^n}\phi(x){\mu}[t](dx)dt\\ =\int_0^Tdt_*\psi(t_*)\int_{\mathbb{R}^n}\mu[t_*](dx_*) \int_{\mathcal{S}[V,t_*,x_*,\mu]}\langle\nabla \phi(x(t_*)),\dot{x}(t_*)\rangle \chi_{t_*,x_*}(d(x(\cdot),z(\cdot))). \end{multline*}

Denote $$b(t_*,x_*)\triangleq\int_{\mathcal{S}[V,t_*,x_*,\mu]}\dot{x}(t_*) \chi_{t_*,x_*}(d(x(\cdot),z(\cdot))).$$ We have that $b(t_*,x_*)\in  \mathrm{proj}_1(\overline{\mathrm{co}}\mathcal{T} V(t_*,x_*,{\mu})\cap E^-(t_*,x_*,\mu[t]))$, and
$$-\int_0^T\psi'(t)\int_{\mathbb{R}^n}\phi(x){\mu}[t_*](dx)dt_=\int_0^T\psi(t)\int_{\mathbb{R}^n} \langle\nabla \phi(x),b(t,x)\rangle {\mu}[t](dx)dt. $$
\end{proof}

We say that $V$ is a classical solution of (\ref{HJB}) for a given external field $\mu$ if $V$ is differentiable and satisfies equation (\ref{HJB}) at any position $(t,x)\in \ir$. If there exists $\partial H/\partial p$ and $V$ is differentiable we say that $\mu$ is a weak solution of (\ref{meanfield}) if for all  $\phi\in C^1_b(\mathbb{R}^n)$ and $\psi\in C^1_b([0,T])$ such that $\psi(0)=\psi(T)=0$ the following equality is valid:
\begin{multline*}-\int_{0}^T\int_{\mathbb{R}^n}\psi'(t)\phi(x)\mu[t](dx)dt\\=\int_{0}^T\int_{\mathbb{R}^n}\psi(t)\left\langle \frac{\partial H}{\partial p}(t,x,\mu[t],\nabla V(t,x)),\nabla\phi(x)\right\rangle\mu[t](dx)dt. \end{multline*}

\begin{Pl}
Assume that $({V},{\mu})$ is a minimax solution to  system (\ref{HJB}),(\ref{meanfield}), there exists $\partial H/\partial p$ and $V$ is differentiable. Then
${V}$ is a classical solution of equation (\ref{HJB}) and ${\mu}$ satisfies equation (\ref{meanfield}) in the weak sense.
\end{Pl}
\begin{proof}
Note that the first statement of the Proposition is proved in \cite[\S 2.4]{Subb_book}.
If ${V}$ is differentiable, then
$$d^+V(t,x,1,\xi)=d^-V(t,x,1,\xi)=\frac{\partial V}{\partial t}+\langle \nabla V,\xi\rangle. $$ Since $V$ is a minimax solution of (\ref{HJB}), we have that there exists $u_*\in P$ such that $$\frac{\partial V}{\partial t}+\langle \nabla V,f(t,x,\mu[t],u_*)\rangle-g(t,x,\mu[t],u_*))\geq 0,$$ and for all $u\in P$
$$\frac{\partial V}{\partial t}+\langle \nabla V,f(t,x,\mu[t],u)\rangle-g(t,x,\mu[t],u))\leq 0.$$
This means, that $V$ is a classical solution of equation (\ref{HJB}). Moreover, $ \mathcal{T} V(t,x,\mu)=\{f(t,x,\mu[t],u_*),g(t,x,\mu[t],u_*)\}$.
Therefore,  condition (\ref{integral_transform}) takes the form
$$\frac{d}{dt}\int_{\mathbb{R}^n}\phi(x){\mu}[t](dx)=\int_{\mathbb{R}^n}\langle f(t,x,\mu[t],u_*),\nabla\phi(x)\rangle{\mu}[t](dx)\ \ \forall \varphi\in C^1_0(\mathbb{R}^n). $$
 Further, under assumptions of the Proposition
$$\frac{\partial H}{\partial p}(t,x,\mu[t],\nabla V(t,x))=f(t,x,\mu[t],u_*),$$ we obtain that equation (\ref{meanfield}) is valid in the weak sense.
\end{proof}

\begin{Pl} Assume that
\begin{itemize}
\item $H$ is differentiable with respect to $p$;
\item $V$ is a classical solution of equation (\ref{HJB}), and $\nabla V$ is Lipschitz continuous with respect to $x$;
\item  $\mu$ is a weak solution of (\ref{meanfield}).
\end{itemize}
Then $(V,\mu)$ is a minimax solution to  system (\ref{HJB}),~(\ref{meanfield})
\end{Pl}
\begin{proof} Let $x_*(\cdot,x_0)$ be a solution of the initial value problem
$$\dot{x}=\frac{\partial H}{\partial p}(t,x,\mu[t],\nabla V(t,x)), \ \ x(0)=x_0. $$ Further, denote $$z_*(\cdot,x_0)=V(0,x_0)+\int_0^tH^*(t,x_*(t,x_0),\mu[t],\dot{x}_*(t,x_0))dt.$$
Note that $H^*(t,x_*(t,x_0),\mu[t],\dot{x}_*(t,x_0))=-g(t,x_*(t,x_0),\mu[t],\dot{x}_*(t,x_0))$.
We have that $$\mathcal{S}[V,\mu]=\{(x_*(\cdot,x_0),z_*(\cdot,x_0)):x_0\in\mathbb{R}^n\}. $$ Define the measure $\chi$ by the rule
$$\chi( Y)=m_0\{x_0:(x_*(\cdot,x_0),z_*(\cdot,x_0))\in Y\}. $$
Put $\nu[t]=e_t\#\chi$. We have that $\nu$ is a weak solution of the equation
\begin{equation}\label{mfg_classic}
\frac{d}{dt}\nu[t]=\left\langle \frac{\partial H}{\partial p}(t,x,\mu[t],\nabla V(t,x)),\nabla \right\rangle\nu[t].
\end{equation}
%For $$\tilde{b}(t,x)=\frac{\partial H}{\partial p}(t,x,\mu[t],\nabla V(t,x)). $$

By assumption $\mu$  is also a solution of  equation (\ref{mfg_classic}). From \cite[Theorem 2.3 and Proposition 4.2]{Kolokoltsev} we conclude that $\mu=\nu$.
\end{proof}

\section{Games with Infinitely Many Players}\label{sec_infinitely_many}

In this section we introduce the games with infinitely many players in the case when the dynamics and the outcome of each player depend only on the state of the player, his control and the distribution of players' states. First static games with infinitely many players were considered in \cite{Aumann1}, \cite{Aumann2}, \cite{Vind}; the review of games with infinitely many players can be found in \cite{Nash_many_survey}.  The basic constructions of dynamical games with infinitely many players were first proposed in \cite{petrisyan}.

Let $\Omega$ be a set of players. We assume that $\Omega$ is a compact metric space. Denote the metric on $\Omega$ by $d_\Omega$. Further, let $\eta$ be a nonatomic measure on $\Omega$, $x_0:\Omega\rightarrow G_0$ be a continuous function. %We introduce the set $\Omega$ to avoid the the problems with atoms of~$m_0$.

Let the  state of player $\omega$ at time $t$ $x[t,\omega]$ satisfy the equation
\begin{equation}\label{external_field_sys}
\frac{d}{dt}{x}[t,\omega]=f(t,x[t,\omega],\mu[t],u(t)), \ \ x(0)=x_0(\omega), \ \ u(t)\in P, \ \ t\in [0,T].
\end{equation}
Here $\mu[t]$ is a distribution of players' states at time $t$. It is given by the equality
 $%\begin{equation}\label{mu_eq}
 \mu[t]=x[t,\cdot]\#\eta.
$%\end{equation}

If $u$ is a control of player $\omega$, then his outcome is
\begin{equation}\label{payoff}
J[x[\cdot,\omega],u,\mu]\triangleq \sigma(x[T,\omega],\mu[T])-\int_{0}^Tg(t,x[t,\omega],\mu[t],u(t))dt.
\end{equation}
Each player wants to maximize his own payoff.

Having system (\ref{HJB}), (\ref{meanfield}) we can construct the differential game with  dynamics (\ref{external_field_sys}) and outcome (\ref{payoff}) by setting $\Omega\triangleq G_0\times [0,1]$, $\eta\triangleq m_0\times\lambda$, and $x_0(\omega',\omega'')\triangleq\omega'$ for  $\omega=(\omega',\omega'')$, $\omega'\in G_0$, $\omega''\in [0,1]$. Note that the proposed approach admits the case when the initial measure $m_0$ contains atoms.

Note that the maximum in (\ref{payoff}) may not be achieved. To relax problem (\ref{external_field_sys}), (\ref{payoff}) we introduce the following construction proposed in \cite{Chen_1976}. Let $S$ and $R$ be a compact metric spaces. The set $C(S\times R)$ is a separable metric space. Let $\varphi_1,\varphi_2,\ldots$ be dense in $C(S\times R)$. If  $\varkappa\in\mathcal{P}(S\times R)$, then define the weak norm of $\varkappa$ by the rule
$$\|\varkappa\|_{w}\triangleq\sum_{j=1}^\infty\frac{|\langle\varphi_j,\varkappa\rangle|}{2^j(1+\|\varphi_j\|)}. $$ Here $$\langle\varphi,\varkappa\rangle=\int_{S\times R}\varphi(s,r)\varkappa(d(s,r)). $$ The set $\mathcal{P}(S\times R)$ with the norm $\|\cdot\|_w$ is compact. Moreover $\|\varkappa_i-\varkappa\|\rightarrow 0$, as $i\rightarrow\infty$ if and only if $\varkappa_i$ converges to $\varkappa$ in the weak sense.

Now let  $\theta$ be a nonatomic Borel measure on   $S$. Denote $$\Lambda(S,\theta,R)\triangleq\{\varkappa\in\mathcal{P}(S\times R):\varkappa( Y\times R)=\theta( Y)\}.$$ The set $\Lambda(S,\theta,R)$ is also compact.

Recall \cite[Theorem 10.4.6]{Bogachev} that for each $\varkappa\in \Lambda(S,\theta,R)$ there exists a function $h:S\rightarrow\mathcal{P}(R)$ such that for any $\varphi\in C(S\times R)$ the function
\begin{equation}\label{measurability_h}
s\mapsto \int_R\varphi(s,t)h(s)(dr)\mbox{ is measurable}
\end{equation}
and
\begin{equation}\label{diff_kappa}
\int_S\theta(ds)\int_R\varphi(s,r)h(s)(dr)=\int_{S\times R}\varphi(s,r)\varkappa(d(s,r)).
\end{equation}
Below we use the denotation
$$\frac{\partial\varkappa}{\partial \theta}(s,dr)=h(s)(dr). $$

Conversely, if $h$ satisfies condition (\ref{measurability_h}), then there exists a measure $\varkappa\in \Lambda(S,\theta,R)$ such that condition (\ref{diff_kappa}) holds.

Denote $\mathcal{U}=\Lambda([0,T],\lambda,P)$. Here $\lambda$ is Lebesgue measure on $[0,T]$. Elements of $\mathcal{U}$ are control measures. If $\alpha\in \mathcal{U}$, $\mu\in\mathcal{M}$, $\omega\in \Omega$, then the corresponding motion $x[\cdot,\omega,\alpha,\mu]$ is a solution of the equation
$$x(t)=x_0(\omega)+\itp{t}f(\tau,x(\tau),\mu[\tau],u)\alpha(d(\tau,u)). $$
Moreover, the outcome of player $\omega$ playing with the control measure $\alpha$ is
$$J[\omega,\alpha,\mu]=\sigma(x[T,\omega,\alpha,\mu],\mu[T])- \itp{T}g(\tau,x[\tau,\omega,\alpha,\mu],\mu[\tau],u)\alpha(d(\tau,u)). $$

The approach based on control measures is equivalent to the approach based on measure-valued controls proposed by Warga \cite{Warga}. Indeed, if $h=\frac{\partial\alpha}{\partial \lambda}$ is a measure-valued control, then
$x[\cdot,\omega,\alpha,\mu]$ is a solution of initial value problem
$$\dot{x}=\int_Pf(t,x(t),\mu[t],u)h(t,du), \ \ x(0)=x_0(\omega). $$ Analogously,
$$J[\omega,\alpha,\mu]=\sigma(x[T,\omega,\alpha,\mu],\mu[T])- \int_0^{T}\int_Pg(\tau,x[\tau,\omega,\alpha,\mu],\mu[\tau],u)h(t,du)dt. $$

Below we use the construction analogous to the mixed strategies in the theory of static games with the finite number of players. Denote $\mathcal{D}=\Lambda(\Omega,\eta,\mathcal{U})$. The elements of $\mathcal{D}$ are the profile of open--loop strategies, i.e. if $\gamma\in\mathcal{D}$, then for $\eta$-almost every $\omega\in \Omega$ $\partial \gamma/ \partial \eta(\omega,\cdot) $ is a distribution of controls chosen by player $\omega$. %It is an analog of mixed strategy for the static games with finite number of players.

\begin{De}We say that the pair $X[\gamma]=(y,\mu)$ is a process generated by the profile of strategies $\gamma$ if
$y:[0,T]\times \Omega\times \mathcal{U}\rightarrow \mathbb{R}^n$ and $\mu\in\mathcal{M}$ are such that
$y[\cdot,\omega,\alpha]=x[\cdot,\omega,\alpha,\mu]$   for all $\alpha\in\mathcal{U}$, $\omega\in \Omega$ and $\mu[t]=y[t,\cdot,\cdot]\#\gamma$.
\end{De}

Recall that the equality $\mu[t]=y[t,\cdot,\cdot]\#\gamma$ can be rewritten in the following way:
$$
  \int_{\mathbb{R}^n}\phi(x)\mu[t](dx)=\iou\phi(y[t,\omega,\alpha])\gamma(d(\omega,\alpha)), \ \ \forall\phi\in C_b(\mathbb{R}^n).$$

Put $$\mathcal{M}'\triangleq\{\mu\in\mathcal{M}:\mu[t]\in\mathcal{P}(G), \ \ W(\mu(t',\cdot),\mu(t'',\cdot))\leq K|t'-t''|\}. $$ Here $K$ is defined by (\ref{K_def}). The set $\mathcal{M}'$ is  convex and compact.

\begin{Pl}\label{pl_process} For each profile of strategies $\gamma$ there exists a unique process $X[\gamma]$.
\end{Pl}
\begin{proof}Let us introduce the operator $A:\mathcal{M}'\rightarrow\mathcal{M}'$.
Put $A(\mu)[t]\triangleq x[t,\cdot,\cdot,\mu]\#\gamma$, i.e.
$\nu=A(\mu)$ if
\begin{equation}\label{A_def}
\int_{\mathbb{R}^n}\phi(x)\nu[t](dx)=\iou\phi(x[t,\omega,\alpha,\mu])\gamma(d(\omega,\alpha)),  \ \ \forall \phi\in C_b(\mathbb{R}^n).
\end{equation} %This formula defines the measure $\nu[t]$  in the unique way.

Let us show that $\nu=A(\mu)\in\mathcal{M}'$. Let $\phi(x)=0$ on $G$. Then
$$\int_{\mathbb{R}^n}\phi(x)\nu[t](dx)=0. $$ This means that ${\rm supp}(\nu[t])\subset G$. Further, for $\phi\geq 0$
$$\int_{\mathbb{R}^n}\phi(x)\nu(dx)\geq 0, \ \ \mbox{and} $$
$$\int_{\mathbb{R}^n}1\nu[t](dx)=1. $$ Hence, $\nu[t]\in\mathcal{P}(G).$

Now let $t',t''\in [0,T]$, $t''>t'$. We have that
$$x[t'',\omega,\alpha,\mu]=x[t',\omega,\alpha,\mu]+\ittp{t'}{t''} f(\tau,x[\tau,\omega,\alpha,\mu],\mu[\tau],u)\alpha(d(\tau,u)). $$
Thus,
$$\|x[t'',\omega,\alpha,\mu]-x[t',\omega,\alpha,\mu]\|\leq K|t''-t'|. $$
Further,
\begin{multline*}W(\nu(t'',\cdot),\nu(t',\cdot))= \sup\Bigl\{\int_{\mathbb{R}^n}\phi(x)\nu(t'',dx)-\int_{\mathbb{R}^n}\phi(x)\nu(t',dx):\phi\in{\rm Lip}_1\Bigr\}\\
=\sup\Bigl\{\iou [\phi(x[t'',\omega,\alpha,\mu])-\phi(x[t',\omega,\alpha,\mu])]\gamma(d(\omega,\alpha)):\phi\in{\rm Lip}_1\Bigr\} \\\leq
K|t''-t'|.
\end{multline*}
Hence, the operator $A$ is well-defined.

Now let us show that $A$ is continuous. Let $\mu_1,\mu_2\in \mathcal{M}'$. We have that
\begin{multline*}\|x[t,\omega,\alpha,\mu_1]-x[t,\omega,\alpha,\mu_2]\|\\\leq \int_{0}^t L_{f,x}\|x[\tau,\omega,\alpha,\mu_1]-x[\tau,\omega,\alpha,\mu_2]\|d\tau+ \int_{0}^tL_{f,m}W(\mu_1[\tau],\mu_2[\tau])d\tau. \end{multline*}
Using Gronwall's inequality we get
\begin{equation}\label{x_estima_gronwall}
\|x[t,\omega,\alpha,\mu_1]-x[t,\omega,\alpha,\mu_2]\|\leq e^{L_{f,x}T}L_{f,m}\int_{0}^t W(\mu_1[\tau],\mu_2[\tau])d\tau.
\end{equation}
 Denote $C_1=e^{L_{f,x}T}L_{f,m}$. Since
\begin{multline*}
W(A(\mu_1)[t],A(\mu_2)[t])\\=\sup\Bigl\{ \iou(\phi(x[t,\omega,\alpha,\mu_1])-\phi(x[t,\omega,\alpha,\mu_2]))\gamma(d(\omega,\alpha)):\phi\in{\rm Lip}_1\Bigr\}, \end{multline*}
we have that
\begin{equation}\label{W_A_estim_t}
  W(A(\mu_1)[t],A(\mu_2)[t])\leq C_1 \int_{0}^tW(\mu_1[\tau],\mu_2[\tau])d\tau.
\end{equation}

Hence,
$$\mathcal{W}(A(\mu_1),A(\mu_2))\leq C_1T\mathcal{W}(\mu_1,\mu_2). $$

Since $\mathcal{M}'$ is compact and $A:\mathcal{M}'\rightarrow\mathcal{M}$ is continuous, $A$ admits a fixed point $\mu^*$. The pair $(y^*,\mu^*)$ with $y^*:[0,T]\times \mathcal{U}\times\Omega\rightarrow \mathcal{R}^n$ given by the equality $y^*[\cdot,\omega,\alpha]=x[\cdot,\omega,\alpha,\mu^*]$ is a process generated by $\gamma$.

Now let us show the uniqueness of $X[\gamma]$. Note that if $(y,\mu)$ is a process generated by $\gamma$, then $\mu=A(\mu)$. Let $\mu_1$ and $\mu_2$ be fixed points of $A$. Let $\vartheta$ be a maximal time  such that $\mu_1[t]=\mu_2[t]$ for $t\in[0,\vartheta]$. If $\vartheta<T$ choose a positive number $\delta$ such that $\delta< \min\{T-\vartheta, 1/C_1\}$. It follows from (\ref{W_A_estim_t}) that
$$\sup_{t\in [\theta,\vartheta+\delta]}W(\mu_1[t],\mu_2[t])\leq C_1\delta\sup_{t\in [\vartheta,\vartheta+\delta]}W(\mu_1[t],\mu_2[t]). $$ Therefore, $W(\mu_1[t],\mu_2[t])=0$ for $t\in [\vartheta,\vartheta+\delta]$. This contradicts with the choice of~$\vartheta$.
\end{proof}

\begin{De} We say that $\hat{\gamma}$ is a Nash equilibrium profile if for $(\hat{y},\hat{\mu})=X[\hat{\gamma}]$,  $\eta$-almost all $\omega\in\Omega$ and for all  $\varrho\in\mathcal{P}(\mathcal{U})$ the following inequality if fulfilled
\begin{equation}\label{nash_ineq}
\int_{\mathcal{U}}J[\omega,\alpha,\hat{\mu}]\varrho(d\alpha)\leq \int_{\mathcal{U}}J[\omega,\alpha,\hat{\mu}]\frac{\partial\hat{\gamma}}{\partial \eta}(\omega,d\alpha).
\end{equation}
\end{De}

Denote  $$\mathcal{E}(\omega,\mu)={\rm Argmax}\{J[\omega,\alpha,\mu]:\alpha\in\mathcal{U}\}, \ \ \mathcal{E}(\mu)=\{(\omega,\alpha):\alpha\in\mathcal{E}(\omega,\mu)\}.$$
The set $\mathcal{E}(\mu)$ is closed.

\begin{Pl}The profile $\hat{\gamma}$ is a Nash equilibrium, if and only if  ${\rm supp}(\hat{\gamma})\subset\mathcal{E}(\hat{\mu}) $ for $(\hat{y},\hat{\mu})=X[\hat{\gamma}]$.
\end{Pl}
\begin{proof}First assume that $\hat{\gamma}$ is a Nash equilibrium.
Let $\Omega^+[\hat{\gamma}]$ be a set of $\omega\in \Omega$ such that inequality (\ref{nash_ineq}) holds. Since $\hat{\gamma}$ is a Nash equilibrium, $\eta(\Omega\setminus \Omega^+[\hat{\gamma}])=0$.
For each $\omega\in \Omega^+[\gamma]$ the  following inclusion holds $${\rm supp}\left(\frac{\partial\hat{\gamma}}{\partial \eta}(\omega,\cdot)\right)\subset\mathcal{E}(\omega,\hat{\mu}).$$

Let $\varphi\in C(\Omega\times\mathcal{U})$ be such that $\varphi(\omega,\alpha)=0$ for $(\omega,\alpha)\in\mathcal{E}(\hat{\mu})$.

We have that $$\iou\varphi(\omega,\alpha)\hat{\gamma}(d(\omega,\alpha))= \int_{\Omega^+}\eta(d\omega)\int_\mathcal{U}\varphi(\omega,\alpha)\frac{\partial\hat{\gamma}}{\partial\eta}(\omega,d\alpha)=0. $$ Therefore, ${\rm supp}(\hat{\gamma})\subset\mathcal{E}(\hat{\mu})$.

Now assume that ${\rm supp}(\hat{\gamma})\subset\mathcal{E}(\hat{\mu}) $ for $(\hat{y},\hat{\mu})=X[\hat{\gamma}]$. This mean that for $\eta$-almost all $\omega\in\Omega$ ${\rm supp}(\partial\hat{\gamma}/\partial\eta(\omega,\cdot))\subset\mathcal{E}(\omega,\mu). $ Therefore, for $\eta$-almost all $\omega\in\Omega$, all $\alpha\in {\rm supp}(\frac{\partial\hat{\gamma}}{\partial\eta}(\omega,\cdot))$ and all $\beta\in\mathcal{U}$
$$%\begin{equation}\label{J_beta_alpha}
J[\omega,\beta,\hat{\mu}]\leq J[\omega,\alpha,\hat{\mu}].
$$%\end{equation}
Integration of this inequality with respect to the measure $\frac{\partial\hat{\gamma}}{\partial\eta}(\omega,d\alpha)$ and with respect to the measure $\varrho(d\beta)$ gives inequality (\ref{nash_ineq}).
\end{proof}

\begin{Th}\label{th_equilibrium} There exists a Nash equilibrium profile of strategies.
\end{Th}
\begin{proof}
To prove the existence of the Nash equilibrium profile we construct the multivalued map such that its fixed points are Nash equilibria in the game with infinitely many players.

First let us define the function $B:\mathcal{M}'\times \mathcal{D}\rightarrow \mathcal{M}'$ by the following rule: $\nu=B(\mu,\gamma)$ if $\nu[t]=x[t,\cdot,\cdot,\mu]\#\gamma$, i.e.
\begin{equation}\label{B_def}
  \int_{\mathbb{R}^n}\phi(x)\nu[t](dx)= \iou\phi(x[t,\omega,\alpha,\mu])\gamma(d(\omega,\alpha)) \ \ \forall\phi\in C_b(\mathbb{R}^n).
\end{equation}
The function $B$ is well-defined (the proof is analogous to the proof of correctness of the definition of $A$ in the proof of Proposition \ref{pl_process}).

Now let us show that $B$ is continuous. Assume the converse. This means that there exist   $\mu_k\rightarrow\mu$, $\gamma_k\rightarrow \gamma$, as $k\rightarrow\infty$ such that
$$\lim_{k\rightarrow\infty}\mathcal{W}(B(\mu_k,\gamma_k),B(\mu,\gamma))=\lim_{k\rightarrow\infty}\sup_{t\in [0,T]}W(B(\mu_k,\gamma_k)[t],B(\mu,\gamma)[t])>0.$$

By extracting subsequences we can assume that there exists $t\in [0,T]$ satisfying the property
\begin{equation}\label{B_discont}
\lim_{k\rightarrow\infty}W(B(\mu_k,\gamma_k)[t],B(\mu,\gamma)[t])>0.
\end{equation}
Recall that
\begin{multline*}
  W(B(\mu_k,\gamma_k)[t],B(\mu,\gamma)[t])\\= \sup\Bigl\{\int_{\mathbb{R}^n}\phi(x)B(\mu_k,\gamma_k)[t](dx)-\int_{\mathbb{R}^n}\phi(x)B(\mu,\gamma)[t](dx):\phi\in{\rm Lip}_1\Bigr\}.
\end{multline*}

Therefore, (\ref{B_discont}) can be reformulated in the following way: there exists a sequence $\{\phi_k\}\subset {\rm Lip}_1\cap C(G)$ such that
$$\lim_{k\rightarrow\infty}\left[\int_G\phi_k(x)B(\mu_k,\gamma_k)[t](dx)-\int_G\phi_k(x)B(\mu,\gamma)[t](dx)\right]>0. $$

Without loss of generality we can assume that there exists a function $\phi_*\in {\rm Lip}_1\cap C(G)$ such that $\|\phi_{k}-\phi_*\|\rightarrow 0$, as $k\rightarrow\infty$.

We have that
\begin{multline*}\int_{G}\phi_{k}(x)B(\mu_{k},\gamma_{k})[t](dx) -\int_{G}\phi_{k}(x)B(\mu,\gamma)[t](dx) \\=\iou\phi_{k}(x[t,\omega,\alpha,\mu_{k}])\gamma_{k}(d(\omega,\alpha))- \iou\phi_{k}(x[t,\omega,\alpha,\mu])\gamma(d(\omega,\alpha))\\=
\iou\phi_{k}(x[t,\omega,\alpha,\mu_k])\gamma_{k}(d(\omega,\alpha))
-\iou\phi_k(x[t,\omega,\alpha,\mu])\gamma_{k}(d(\omega,\alpha))\\+
\iou\phi_{k}(x[t,\omega,\alpha,\mu])\gamma_{k}(d(\omega,\alpha))
-\iou\phi_*(x[t,\omega,\alpha,\mu])\gamma_{k}(d(\omega,\alpha))\\
+\iou\phi_*(x[t,\omega,\alpha,\mu])\gamma_{k}(d(\omega,\alpha)) -\iou\phi_*(x[t,\omega,\alpha,\mu])\gamma(d(\omega,\alpha))\\+
\iou\phi_*(x[t,\omega,\alpha,\mu])\gamma(d(\omega,\alpha))
-\iou\phi_{k}(x[t,\omega,\alpha,\mu])\gamma(d(\omega,\alpha))
.
\end{multline*}
From this and (\ref{x_estima_gronwall}) we get the inequality
\begin{multline*}
  W(B(\mu_{k},\gamma_{k})[t],B(\mu,\gamma)[t])\leq  C_1T\mathcal{W}(\mu_{k},\mu)+2\|\phi_{k}-\phi_*\|\\+%\sup_{\phi\in {\rm Lip}_1}
  \left[
\iou\phi_*(x[t,\omega,\alpha,\mu])\gamma_k(d(\omega,\alpha))- \iou\phi_*(x[t,\omega,\alpha,\mu])\gamma(d(\omega,\alpha))\right].
\end{multline*}
Since  for all $\varphi\in C(\Omega\times\mathcal{U})$ $$\iou\varphi(\omega,\alpha)\gamma_k(d(\omega,\alpha))\rightarrow \iou\varphi(\omega,\alpha)\gamma(d(\omega,\alpha))$$ we obtain that
$$\lim_{k\rightarrow \infty}W(B(\mu_{k},\gamma_{k})[t],B(\mu,\gamma)[t])=0. $$ This contradicts with (\ref{B_discont}). Thus, $B$ is continuous.

Now let  $\mathcal{F}(\mu)$ be a set of all profiles $\xi\in\mathcal{D}$ such that ${\rm supp}(\xi)\subset\mathcal{E}(\mu)$.
The set $\mathcal{F}(\mu)$ is nonempty. Indeed, since the correspondence $\omega\mapsto\mathcal{E}(\omega,\mu)$ is upper semicontinuous, $\mathcal{E}(\omega,\mu)$ admits a measurable selector    $\alpha_*(\omega,\mu)$. Define the measure $\xi_*$ by the rule
$$\iou\varphi(\omega,\alpha)\xi_*(d(\omega,\alpha))=\int_\Omega\varphi(\omega,\alpha_*(\omega,\mu))\eta(d\omega). $$ We have that ${\rm supp}(\xi_*)\in\mathcal{E}(\mu)$. Thus, $\xi_*\in \mathcal{F}(\mu)$.

Moreover, $\mathcal{F}(\mu)$ is convex.

Further we shall prove the closeness of the graph of the mapping $\mathcal{F}$.

Let $\mu_k\rightarrow\mu$, $\xi_k\rightarrow\xi$, $\xi_k\in\mathcal{F}(\mu_k)$. We shall show that $\xi\in\mathcal{F}(\mu)$. It suffices to show that ${\rm supp}(\xi)\subset\mathcal{E}(\mu)$. Assume the converse. Let there exist a set $Q\subset\Omega\times\mathcal{U}$ such that $Q\cap \mathcal{E}(\mu)=\varnothing$ and $\xi(Q)>0$. We can choose $Q$ to be compact. Since $\mathcal{E}(\mu)$ is compact also, there exists $\varepsilon$ such that the set
$Q_\varepsilon=\{(\omega,\alpha):d((\omega,\alpha),Q)\leq\varepsilon\}$ doesn't intersect with $\mathcal{E}(\mu)$. Here $d$ is a distance between $(\omega,\alpha)$ and $Q$:
$$d((\omega,\alpha),Q)=\min\{d_\Omega(\omega,\omega')+\|\alpha-\alpha'\|_w:(\omega',\alpha')\in Q\}. $$

There exists $N$ such that $Q_\varepsilon\cap\mathcal{E}(\mu_k)=\varnothing$ for all $k>N$. In the contrary case   there exists a sequence $\{k_l\}$ such that $(\omega_{k_l},\alpha_{k_l})\in Q_\varepsilon\cap\mathcal{E}(\mu_{k_l})$. We can assume that $(\omega_{k_l},\alpha_{k_l})\rightarrow (\omega^*,\alpha^*)$, as $l\rightarrow\infty$. Since the dependence $\mu\mapsto \mathcal{E}(\mu)$ is upper semicontinuous we have that $(\omega^*,\alpha^*)\in Q_\varepsilon\cap \mathcal{E}(\mu)$. This contradicts with the emptiness of $Q_\varepsilon\cap \mathcal{E}(\mu)$.

Now let $\varphi\in C({\Omega}\times\mathcal{U})$ be such that $\varphi=1$ on $Q$ and $\varphi=0$ outside $Q_\varepsilon$.
We have that $$\iou\varphi(\omega,\alpha)\xi_k(d(\omega,\alpha))=0, \ \ k>N,\mbox{ and }\iou\varphi(\omega,\alpha)\xi(d(\omega,\alpha))\geq \xi(Q)>0. $$ This contradicts with the assumption of weak convergence of  the sequence  $\{\xi_k\}$ to $\xi$. Thus, $\xi\in\mathcal{F}(\mu)$.

Define the multi-valued map $\mathcal{G}:\mathcal{M}'\times\mathcal{U}\multimap\mathcal{M}'\times\mathcal{U}$ by the rule
$$\mathcal{G}(\mu,\gamma)\triangleq\{(\nu,\xi):\nu=B(\mu,\gamma), \ \ \xi\in\mathcal{F}(\mu)\}. $$ The map $\mathcal{G}$ is upper semicontinuous. The set $\mathcal{M}'\times\mathcal{U}$ is compact. Therefore, by Kakutani--Fan--Glicksberg theorem there exists a fixed point of $\mathcal{G}$. Denote it by $(\hat{\mu},\hat{\gamma})$. Since $\hat{\mu}=B(\hat{\mu},\hat{\gamma})$ we have that the pair $(\hat{y},\hat{\mu})=X[\hat{\gamma}]$. Here $\hat{y}$ is such that $\hat{y}[\cdot,\omega,\alpha]=x[\cdot,\omega,\alpha,\hat{\mu}]$. Further, since ${\rm supp}(\hat{\gamma})\subset\mathcal{E}(\hat{\mu})$ inequality (\ref{nash_ineq}) holds.
\end{proof}

\section{Existence of the Solution of MFG System}\label{sec_dynamic_programming}

\begin{Th}\label{th_existence} There exists a minimax solution $(V,\mu)$ of  system~(\ref{HJB}),~(\ref{meanfield}).
\end{Th}

Before we prove the theorem, let us fix some notation.
If $\mu$ is an external field, $\alpha\in\mathcal{U}$, $(t_*,x_*)\in\ir$, then denote by $x(\cdot,t_*,x_*,\alpha,\mu)$ the solution of the equation
$$x(t)=x_*+\int_{[t_*,t]\times P}f(\tau,x(\tau),\mu[\tau],u)\alpha(d(\tau,u)). $$ In addition, put
$$w(t,t_*,x_*,\alpha,\mu)=-\int_{[t_*,t]\times P}g(t,x(\tau,t_*,x_*,\alpha,\mu),\mu[t],u)\alpha(d(\tau,u)). $$

%If $\hat{\gamma}$ is a Nash equilibrium, and $X[\hat{\gamma}]=(\hat{y},\hat{\mu})$, then put $$\hat{s}[t,\omega,\alpha]=-\int_0^tg(t,y[\tau,\omega,\alpha],\hat{\mu}[\tau],u)\alpha(d(\tau,u)).$$ Note that $J[\omega,\alpha,\hat{\mu}]=\sigma(\hat{y}[T,\omega,\alpha],\hat{\mu}[T])+\hat{s}[T,\omega,\alpha]$.

Further, if $\hat{\gamma}$ is a Nash equilibrium with $X[\hat{\gamma}]=(\hat{y},\hat{\mu})$ denote $$\hat{x}(\cdot,t_*,x_*,\alpha)=x(\cdot,t_*,x_*,\alpha,\hat{\mu}), \  \hat{w}(t,t_*,x_*,\alpha)={w}(t,t_*,x_*,\alpha,\hat{\mu}).$$

Define the value function $V$ by the rule
\begin{equation}\label{V_def}
V(t_*,x_*)=\max\{\sigma(\hat{x}(T,t_*,x_*,\alpha),\hat{\mu}[T])+\hat{w}(T,t_*,x_*,\alpha):\alpha\in\mathcal{U}\}. \end{equation}
The number $V(t_*,x_*)$ is an optimal outcome of the sampling player placed in the initial time $t_*$ at the position $x_*$ in the case when the external field is $\hat{\mu}$.

Since the initial state of player $\omega$ is $x_0(\omega)$ we have that $V(0,x_0(\omega))=\max\{J[\omega,\alpha,\hat{\mu}]:\alpha\in\mathcal{U}\}$.

%Put $${\rm Sol}\triangleq\{(\hat{x}[t,0,x_*,\alpha],V(0,x_*)+\hat{w}[t,0,x_*,\alpha]):\alpha\in\mathcal{U}, x_*\in G_0\}.$$

%The set ${\rm Sol}$ is a compact subsets of $C[t_*,T]$.

Define the measure $\chi$ on $C([0,T],\mathbb{R}^{n+1})$ by the rule: for any measurable set $Y\subset C([0,T],\mathbb{R}^{n+1})$
$$\chi(Y)=\hat{\gamma}\{(\omega,\alpha): (\hat{y}[\cdot,\omega,\alpha],V(\cdot,y[\cdot,\omega,\alpha])\in Y\}. $$ The measure $\chi$ is used for examination whether the pair $(V,\hat{\mu})$ is a generalized solution of system (\ref{HJB}), (\ref{meanfield}).

\begin{proof}[Proof of Theorem \ref{th_existence}]
First, let us recall that for given system system (\ref{HJB}), (\ref{meanfield}) we can define the differential game with infinitely many players with  dynamics (\ref{external_field_sys}) and the payoff given by (\ref{payoff}). To do this put $\Omega=G_0\times [0,1]$, $\eta=m_0\times \lambda$. If $\omega=(\omega',\omega'')$, $\omega'\in G_0$, $\omega''\in [0,1]$, then put $x_0(\omega)=\omega'$. There exists an equilibrium $\hat{\gamma}$, let $(\hat{y},\hat{\mu})=X[\hat{\gamma}]$.

Note that $V$ defined by rule (\ref{V_def}) is the unique solution of equation (\ref{HJB}) for $\mu=\hat{\mu}$. The support of the measure $\chi$ is the set of solutions to (\ref{lower_in}) viable in the graph of $V$. The measure $\hat{\mu}[t]$ is an image of $\chi$ by the map $e_t$.

\end{proof}

\section{Approximate Equilibrium in the Game with Finite Number of Players}\label{sec_approximate}
In this section we work with the additional assumption: $\sigma$ and $g$ are Lipschitz continuous with respect to phase variable and measure: i.e. there exist constants $L_{\sigma,x}$, $L_{\sigma,m}$, $L_{g,x}$ and $L_{g,m}$ such that for all $t\in[0,T]$, $x',x''\in G$, $m,m''\in \mathcal{P}(G)$, $u\in P$
$$|\sigma(x',m')- \sigma(x'',m'')|\leq L_{\sigma,x}\|x'-x''\|+L_{\sigma,m}W(m',m''),$$
$$|g(t,x',m',u)- g(t,x'',m'',u)|\leq L_{g,x}\|x'-x''\|+L_{g,m}W(m',m'').$$

If $\mathbf{x}=(x_{N,0}^1,\ldots,x_{N,0}^{N})\in (G_0)^N$, then denote $$\delta_{\mathbf{x}}^N=\frac{1}{N}(\delta_{x_{N,0}^1}+\ldots+\delta_{x_{N,0}^N}). $$ Here $\delta_x$ denote the Dirac measure concentrated at $x$.
\begin{Lm}\label{lm_measure_repres}
There exist measures $m_N^1,\ldots,m_N^N$ such that
\begin{enumerate}
  \item $m_0=m_N^1+\ldots+m_N^N$;
  \item $m_N^i(G_0)=1/N$;
  \item $$W(m_0,\delta_{\mathbf{x}}^N)=\sum_{i=1}^N\int_{G_0}\|x-x_{N,0}^i\|m_N^i(dx). $$
\end{enumerate}
\end{Lm}
\begin{proof}
We have that $$W(m_0,\delta^N_{\mathbf{x}})=\inf\left\{\int_{\mathbb{R}^n\times\mathbb{R}^n}\|x'-x''\|\pi(d(x',x'')):\pi\in \Pi(m_0,\delta_{\mathbf{x}}^N)\right\}. $$ Since ${\rm supp}(m_0)$ and ${\rm supp}(\delta_{\mathbf{x}}^N)$ are subsets of $G_0$, the supports of all measures of $\Pi(m_0,\delta_{\mathbf{x}}^N)$ lies in $G_0\times G_0$. Recall that $G_0$ is compact. Therefore, there exists a measure $\bar{\pi}$ such that
\begin{equation}\label{wasser_delta}
W(m_0,\delta^N_{\mathbf{x}})=\int_{G_0\times G_0}\|x'-x''\|\bar{\pi}(d(x',x'')).
\end{equation}

Since $\bar{\pi}(G_0\times  Y)=\delta_{\mathbf{x}}^N( Y)$ we have that there exists a function $h:G_0\times\mathcal{B}(G_0)$ such that for any $\Upsilon\subset G_0\times G_0$
 \begin{equation}\label{pi_h_repres}
 \bar{\pi}(\Upsilon)=\int_{\Upsilon_2}h(y,\Upsilon_1(y))\delta^N_{\mathbf{x}}(dy).
\end{equation}
Here $$\Upsilon_1(y)=\{x:\exists y\ \ (x,y)\in\Upsilon\}\ \ \Upsilon_2=\{y:\exists x\ \ (x,y)\in \Upsilon\}. $$

Denote $$m_N^i(\cdot)=\frac{1}{N}h(x_{N,0}^i,\cdot).$$ From (\ref{pi_h_repres}) we have that the first and second statements of the Lemma are fulfilled.

Further,
$$W(m_0,\delta^N_{\mathbf{x}})= \int_{G_0}\int_{G_0}\|x'-x''\|h(x'',dx')\delta_{\mathbf{x}}^N(dx''). $$ Therefore, the third statement of the Lemma is also fulfilled.
\end{proof}

Recall that
$$\mathcal{S}[V,\hat{\mu}]=\{(x(\cdot),z(\cdot)):\dot{x}\in E^-(t,x,\hat{\mu}[t]), \ \ z(t)=V(t,x(t))\}, $$
$$\mathcal{S}[V,t_*,x_*,\hat{\mu}]=\{(x(\cdot),z(\cdot))\in \mathcal{S}[V,\hat{\mu}]:x(t_*)=x_*\}.$$

As it was mentioned above (see the proof of Proposition \ref{pl_infinitisimal})  there exists a system of measures $\chi_{0,x_*}$ such that for any $\varphi\in C(\mathcal{S}[V,\hat{\mu}])$
\begin{multline*}\int_{\mathcal{S}[V,\hat{\mu}]}\varphi(x(\cdot),z(\cdot))\chi(d(x(\cdot),z(\cdot)))\\= \int_{G_0}m_0(dx_*)\int_{\mathcal{S}[V,0,x_*,\hat{\mu}]}\varphi(x(\cdot),z(\cdot))\chi_{0,x_*}(d(x(\cdot),z(\cdot))). \end{multline*}

Define the measure $\chi_N^i$ by the rule: for all $\varphi\in C(\mathcal{S}[V,\hat{\mu}])$
\begin{multline*}\int_{\mathcal{S}[V,\hat{\mu}]}\varphi(x(\cdot),z(\cdot))\chi_N^i(x(\cdot),z(\cdot))\\= \int_{G_0}m_N^i(dx_*)\int_{\mathcal{S}[V,0,x_*,\hat{\mu}]}\varphi(x(\cdot),z(\cdot)) \chi_{x_*}(d(x(\cdot),z(\cdot))).\end{multline*}

Note that $\chi=\chi_N^1+\ldots+\chi_N^N$.

Put $\mu_N^i[t]=e_t\#\chi_N^i$.
The external field $\hat{\mu}$ is a sum of fields $\mu_N^i$, i.e. $$\hat{\mu}[t]={\mu}_N^1[t]+\ldots+{\mu}_N^N[t].$$

Now let us consider the $N$ player game.
Having an external field $\nu_N$ and a control measure $\alpha\in\mathcal{U}$ one can consider the motion of the player $i$  given by
$${x}_N^i(t,\alpha,\nu_N)=x(t,0,x_{N,0}^i,\alpha,\nu_N). $$ The outcome of  player $i$ is
\begin{multline*}J_N^i(\alpha,\nu_N)=\sigma(x_N^i(T,\alpha,\nu_N),\nu_N[T])-\itp{T} g(t,x_N^i(t,\alpha,\nu_N),\nu_N[t],u)\alpha(d(t,u))\\=\sigma(x(T,0,x_{N,0}^i,\alpha,\nu_N),\nu_N[T])+ w(T,0,x_{N,0}^i,\alpha,\nu_N).                           \end{multline*}

Below we assume that players use random open--loop strategies, i.e. player $i$ chooses the distribution of control measures $\varrho^i_N\in \mathcal{P}(\mathcal{U})$. Denote  $\varrho_N=(\varrho_N^1,\ldots,\varrho_N^N)$. If players use random strategies, then the distribution of the  states of player $i$ is given by the measure $\nu_N^i(t,\cdot,\varrho_N)$ satisfying the condition
$$\int_{\mathbb{R}^n}\phi(x)\nu_N^i(t,dx,\varrho_N) =\frac{1}{N}\int_{\mathbb{R}^n}\phi\left(x_N^i\left(t,\alpha,\nu_N(\cdot,\cdot,\varrho_N)\right)\right) \varrho^i_N(d\alpha).$$ Here we denote $$ \nu_N(t,\cdot,\varrho_N)=\sum_{i=1}^n\nu_N^i(t,\cdot,\varrho_N). $$  The outcome of player $i$ is
$$\bar{J}_N^i(\varrho_N)=\int_\mathcal{U}J_N^i(\alpha,\nu_N(\cdot,\cdot,\varrho_N))\varrho^i_N(d\alpha). $$

The existence of the motions $x_N^i(\cdot,\alpha,\nu_N)$ and of distributions of players' states is proved as Proposition \ref{pl_process}.

Now let us introduce $\varepsilon$-Nash equilibrium profile of strategies.
The correspondence which assigns to the motion $(x(\cdot),z(\cdot))\in \mathcal{S}[V,\hat{\mu}]$ the set of control measures $\beta$ such that $$x(\cdot)=\hat{x}(\cdot,0,x(0),\beta),\ \  z(\cdot)=\sigma(x(T),\hat{\mu})+ \hat{w}(\cdot,0,x(0),\beta)$$ is upper semicontinuous. Let  $\Theta[x(\cdot),z(\cdot)]$  be its measurable selector.

Let $\hat{\varrho}_N^i$ be a random strategy such that $\hat{\varrho}_N^i=N\cdot\Theta\#\chi_N^i$.

Denote $\hat{\xi}^i_N(t,\alpha)=x_N^i(t,\alpha,\nu_N(\cdot,\cdot,\hat{\varrho}_N))$. In addition, set $$\hat{\zeta}^i_N(t,\alpha)\triangleq- \itp{t}g(\tau,\hat{\xi}^i_N(\cdot,\alpha),\hat{\nu}_N(\cdot,\cdot,\hat{\varrho}_N^i),u)\alpha(d(\tau,u)), $$
$$\hat{\nu}_N^i(\cdot,\cdot)\triangleq{\nu}_N^i(\cdot,\cdot,\hat{\rho}_N),\ \ \hat{\nu}_N(\cdot,\cdot)\triangleq{\nu}_N(\cdot,\cdot,\hat{\rho}_N). $$

If player $j$ deviates, i.e. he plays with the strategy $\varrho_N^j$, then denote the corresponding profile of strategies  by $\hat{\varrho}_N|\varrho_N^j$. Put
$$\tilde{\nu}_{N|j}^i(\cdot,\cdot;\varrho^j_N)\triangleq\nu_N^i(\cdot,\cdot,\hat{\varrho}_N|\varrho_N^j),  \ \ \tilde{\nu}_{N|j}(\cdot,\cdot;\varrho^j_N)\triangleq\tilde{\nu}_{N|j}^1(\cdot,\cdot;\varrho^j_N)+\ldots+\tilde{\nu}_{N|j}^N(\cdot,\cdot;\varrho^j_N). $$ Moreover, denote $\tilde{\xi}^i_{N|j}(t,\alpha;\varrho_N^j)\triangleq x_N^i(t,\alpha,\tilde{\nu}_{N|j}(\cdot,\cdot,\tilde{\varrho}_N^j))$. $$\tilde{\zeta}^i_{N|j}(t,\alpha;\varrho_N^j)\triangleq- \itp{t}g(\tau,\tilde{\xi}^i_{N|j}(\cdot,\alpha;\varrho_N^j),\tilde{\nu}_{N|j}^i(\cdot,\cdot,\varrho_N^j),u) \alpha(d(\tau,u)), $$

\begin{Th}\label{th_N_eq}There exist positive constants $\hat{C}_1,\hat{C}_2,\hat{C}_3$such that for all $\varrho_N^j\in\mathcal{P}(\mathcal{U})$
$$\bar{J}_N^i(\hat{\varrho}_N)\geq \bar{J}_N^i(\hat{\varrho}_N|\varrho_N^j)-\hat{C}_1W(m_0,\delta_{\mathbf{x}}^N)-\hat{C}_2 N\max_{i=\overline{1,N}}\int_{G}\|x-x_{N,0}^i\|m_N^i(dx)-\hat{C}_3\frac{1}{N}.$$
 In particular, if
$$W(m_0,\delta_{\mathbf{x}}^N)\rightarrow 0, \ \   N\max_{i=\overline{1,N}}\int_{G}\|x-x_{N,0}^i\|m_N^i(dx)\rightarrow 0,\mbox{ при }N\rightarrow \infty, $$
then for any $\varepsilon$ there exists a number $N_0$ such that for all $N>N_0$ the profile of strategies $\hat{\varrho}_N$ is a $\varepsilon$-Nash equilibrium.
\end{Th}
The proof of the Theorem requires some preliminary lemmas.

Recall the denotation $C_1=L_{f,m}e^{L_{f,x}T}$. Moreover, put $C_2=e^{L_{f,x}T}$.
\begin{Lm}\label{lm_motin_estima} For any external field $\mu\in\mathcal{M}$, and any control measure $\alpha\in\mathcal{U}$ the following estimate holds:
$$\|x(t,0,x_1,\alpha,\mu)-\hat{x}(t,0,x_2,\alpha)\|\leq C_2\|x_1-x_2\|+C_1\int_0^tW(\hat{\mu}[\tau],\mu[\tau])d\tau $$
\end{Lm}
\begin{proof} Denote $x_1(\cdot)=x(\cdot,0,x_1,\alpha,\mu)$, $x_2(\cdot)=\hat{x}(\cdot,0,x_2,\alpha)$.
We have that
\begin{multline*}\|x_1(t)-x_2(t)\|\leq\|x_1-x_2\|\\+\itp{t} \|f(\tau,x_1(\tau),\mu[\tau],u)- f(\tau,x_2(\tau),\hat{\mu}[\tau],u)\|\alpha(d(\tau,u))\\
\leq \|x_1-x_2\|+\int_0^tL_{f,x}\|x_1(\tau)-x_2(\tau)\|d\tau+\int_0^t L_{f,m}W(\mu[\tau],\hat{\mu}[\tau])d\tau. \end{multline*}
Using Gronwall inequality we get that
\begin{equation}\label{x_N_hat_x_estimate}
  \|x_1(t)-x_2(t)\|\leq \left(\|x_1-x_2\|+L_{f,m}\int_0^tW(\mu[\tau],\hat{\mu}[\tau])d\tau\right)e^{L_{f,x}T}
\end{equation}
 Hence the conclusion of the Lemma follows.

\end{proof}

\begin{Lm}\label{lm_measure_nash} There exists a constant $C_3$ such that
$$W(\hat{\mu}[t],\hat{\nu}_N[t])\leq C_3W(m_0,\delta_{\mathbf{x}}^N). $$
\end{Lm}
\begin{proof} The definitions of measures $\chi_N^i$, and $\hat{\nu}_N$ yield the inequality
\begin{multline}\label{W_estimate_N_nash_1}W(\hat{\nu}_N[t],\hat{\mu}[t])= \sup\left\{\int_{\mathbb{R}^n}\phi(x)\hat{\nu}_N[t](dx)-\int_{\mathbb{R}^n}\phi(x)\hat{\mu}[t](dx):\phi\in{\rm Lip}_1\right\}\\=
\left\{\sum_{i=1}^N\int_{\mathcal{S}[V,\hat{\mu}]} (\phi(\hat{\xi}_N^i(t,\Theta[x(\cdot),z(\cdot)]))-\phi(x(t)))\chi_N^i(d(x(\cdot),z(\cdot))):\phi\in{\rm Lip}_1\right\}\\ \leq \sum_{i=1}^N\int_{\mathcal{S}[V,\hat{\mu}]} \|\hat{\xi}_N^i(t,\Theta[x(\cdot),z(\cdot)])-x(t)\|\chi_N^i(d(x(\cdot),z(\cdot)))
\end{multline}

For $\alpha=\Theta[x(\cdot),z(\cdot)]$ and $x_*=x(0)$ we have that
$\xi_N^i(t,\alpha)=x(t,0,x_{N,0}^i,\alpha,\hat{\nu}_N)$, $\hat{x}(t)=\hat{x}(t,0,x_*,\alpha)$. It follows from Lemma \ref{lm_motin_estima} that
$$\|\hat{\xi}_N^i(t,\Theta[x(\cdot),z(\cdot)])-x(t)\| \leq C_2\|x_{N,0}^i-x_*\|+C_1\int_0^tW(\nu_N[\tau],\hat{\mu}[\tau])d\tau.$$

From this and estimate (\ref{W_estimate_N_nash_1}) we get the estimate
$$%\begin{multline*}
W(\hat{\nu}_N[t],\hat{\mu}[t])\leq
C_2\sum_{i=1}^n\int_{G_0}\|x_{N,0}^i-x_*\|m_N^i(dx_*) +C_1\int_0^tW(\nu_N[\tau],\hat{\mu}[\tau])d\tau. $$%\end{multline*}
By Lemma \ref{lm_measure_repres} we conclude that
$$W(\hat{\nu}_N[t],\hat{\mu}[t])\leq C_2W(m_0,\delta_{\mathbf{x}}^N)+C_1\int_0^tW(\nu_N[\tau],\hat{\mu}[\tau])d\tau.  $$ From this and Gronwall's inequality the conclusion of the Lemma follows with $C_3=C_2\exp(C_1T)$.
\end{proof}

\begin{Lm}\label{lm_measure_deviat} There exists a constant $C_4$ such that for all $\varrho_{N}^j$
$$W(\tilde{\nu}_{N|j}(t,\cdot;\varrho_{N}^i),\hat{\mu}[t])\leq C_3W(m_0,\delta_{\mathbf{x}}^N)+\frac{C_4}{N}. $$
\end{Lm}
\begin{proof}The proof is analogous to the proof of the previous Lemma.
As above, we have that
\begin{multline}\label{W_estima_dev}
  W(\tilde{\nu}_{N|j}(t,\cdot;\varrho_N^j),\hat{\mu}[t]) \\
  \leq \sum_{i\neq j}\int_{\mathcal{S}[V,\hat{\mu}]} \|\tilde{\xi}_{N|j}^i(t,\Theta[x(\cdot),z(\cdot)];\varrho_N^j)-x(t)\|\chi_N^i(d(x(\cdot),z(\cdot)))\\+ \sup\left\{\int_{\mathbb{R}^n}\phi(x')\tilde{\nu}_{N|j}^j(t,dx';\varrho_N^j)- \int_{\mathbb{R}^n}\phi(x'')\hat{\mu}_N^j(t,dx''):\phi\in {\rm Lip}_1\right\}.
\end{multline}
Further,
\begin{multline*}\sup\left\{\int_{\mathbb{R}^n}\phi(x')\tilde{\nu}_{N|j}^j(t,dx';\varrho_N^j)- \int_{\mathbb{R}^n}\phi(x'')\hat{\mu}_N^j[t](dx):\phi\in {\rm Lip}_1\right\}\\ =\left\{\int_{G\times G} N(\phi(x')-\phi(x''))\tilde{\nu}_{N|j}^j(t,dx';\varrho_N^j)\hat{\mu}_N^j(t,dx''):\phi\in {\rm Lip}_1\right\}\leq \frac{{\rm diam}(G)}{N}. \end{multline*} Here we use the denotation $${\rm diam}(G)=\sup\{\|x'-x''\|:x',x''\in G\}. $$

From this, (\ref{W_estima_dev}) and Lemma \ref{lm_motin_estima} we get the estimate
$$W(\tilde{\nu}_{N|j}(t,\cdot;\varrho_N^j),\hat{\mu}[t])\leq \frac{{\rm diam}(G)}{N}+C_2W(m_0,\delta_{\mathbf{x}}^N)+C_1\int_0^tW(\tilde{\nu}_{N|j}(\tau,\cdot;\varrho_N^j),\hat{\mu}[\tau])d\tau. $$ Therefore, the  Lemma is valid for $C_4={\rm diam}(G)\exp(C_1T)$.
\end{proof}

\begin{proof}[Proof of Theorem \ref{th_N_eq}]
First we consider the case when all players use the strategies determined by the profile $\hat{\varrho}_N$. Denote $\hat{\rho}_{x}=\Theta\#\chi_{0,x}$. We have that $\hat{\rho}_x$ is a probabilistic measure on $\mathcal{U}$.

From Lemma \ref{lm_motin_estima} it follows that
$$\|\hat{x}(t,0,x_*,\alpha)-\hat{\xi}_N^i(t,\alpha)\|\leq C_2\|x_*-x_{N,0}^i\|+C_1t\mathcal{W}(\hat{\mu},\hat{\nu}_N). $$ From this and Lemma \ref{lm_measure_nash} we get the inequality
\begin{equation}\label{x_xi_nash}
  \|\hat{x}(t,0,x_*,\alpha)-\hat{\xi}_N^i(t,\alpha)\|\leq C_2\|x_*-x_{N,0}^i\|+C_1C_3tW(m_0,\delta_{\mathbf{x}}^N).
\end{equation} Denote $C_5=C_1C_3T$.

Moreover, \begin{multline}\label{z_zeta_nash}\|\hat{w}(T,0,x_*,\alpha)-\hat{\zeta}_N^i(T,\alpha)\|\\\leq \int_{0}^T (L_{g,x}\|\hat{x}(t,0,x_*,\alpha)-\hat{\xi}_N^i(t,\alpha)\|+L_{g,m}W(\hat{\mu}[t],\hat{\nu}_N[t]))dt\\ \leq L_{g,x}C_2T\|x_*-x_{N,0}^i\|+T^2C_1C_3(L_{g,x}+L_{g,m})W(m_0,\delta_{\mathbf{x}}^N). \end{multline}
In addition, denote $C_6=L_{g,x}C_2T$, $C_7=T^2C_1C_3(L_{g,x}+L_{g,m})$.

Therefore,
\begin{multline}\label{main_estimate_nash}\left|\bar{J}_N^j(\hat{\varrho}_N)- N\int_{G_0}m_N^j(dx_*)\int_{\mathcal{U}}\left[\sigma(\hat{x}(T,0,x_*,\alpha),\hat{\mu}[T])+ \hat{w}(T,0,x_*,\alpha)\right])\hat{\rho}_{x_*}(\alpha)\right|\\=
N\int_{G_0}m_N^j(dx_*)
\int_{\mathcal{U}}|\sigma(\hat{\xi}_N^i(T,\alpha),\hat{\mu}[T])+ \hat{\zeta}(T,\alpha)\\-\sigma(\hat{x}(T,0,x_*,\alpha),\hat{\mu}[T])-\hat{w}(T,0,x_*,\alpha)|\hat{\rho}_{x_*}(\alpha)\\ \leq N\int_{G_0}((L_{\sigma,x}C_2+C_6)\|x_*-x_{N,0}^i\|+(L_{\sigma,x}C_5+L_{\sigma,m}C_3+C_7) W(m_0,\delta_{\mathbf{x}}^N))m_N^j(dx_*)\\=(L_{\sigma,x}C_2+C_6)N\int_{G_0}\|x_*-x_{N,0}^i\|m_N^i(dx_*)+(L_{\sigma,x}C_5+L_{\sigma,m}C_3+C_7) W({m}_0,\delta_{\mathbf{x}}^N)). \end{multline}

Now consider the case when the player $j$ deviates.

Since $\hat{\varrho}_{x_*}$ is concentrated on the set of optimal controls, we have that for a sample player starting from $(0,x_*)$, $\alpha\in {\rm supp}(\hat{\rho}_{x_*})$ and all $\beta\in\mathcal{U}$ the following inequality holds true:
  \begin{equation}\label{beta_alpha_ineq}
  \sigma(\hat{x}(T,0,x_*,\beta),\hat{\mu}[T])+\hat{w}[T,0,x_*,\beta]\leq \sigma(\hat{x}(T,0,x_*,\alpha),\hat{\mu}[T])+\hat{w}(T,0,x_*,\alpha).
\end{equation}
Let for each $x_*$ $\rho_{x_*}$ be a probability measure on $\mathcal{U}$. Integrating  inequality (\ref{beta_alpha_ineq}) with respect to $\hat{\rho}_{x_*}(d\alpha)$, $\rho_{x_*}(d\beta)$, and $m_N^j(dx_*)$, we get the inequality
\begin{multline}\label{inequality_rho_rho}
 \int_{G_0}m_N^j(dx_*)\int_{\mathcal{U}}(\sigma(\hat{x}(T,0,x_*,\beta),\hat{\mu}[T])+ \hat{z}(T,0,x_*,\beta))\rho_{x_*}(d\beta)\\ \leq \int_{G_0}m_N^j(dx_*)\int_{\mathcal{U}}(\sigma(\hat{x}(T,0,x_*,\alpha),\hat{\mu}[T])+ \hat{z}(T,0,x_*,\alpha))\hat{\rho}_{x_*}(d\alpha).
\end{multline}

Now, from Lemma \ref{lm_motin_estima} we have that
$$\|\hat{x}(t,0,x_*,\beta)-\tilde{\xi}_{N|j}^j(t,\beta;\varrho_N^j)\|\leq C_2\|x_*-x_{N,0}^i\|+C_1t\mathcal{W}(\hat{\mu},\tilde{\nu}_{N|j}(\cdot,\cdot;\varrho_N^j)). $$

Further, it follows from Lemma \ref{lm_measure_deviat} that
\begin{equation}\label{x_xi_dev}
\|\hat{x}(t,0,x_*,\beta)-\tilde{\xi}_{N|j}^j(t,\beta;\varrho_N^j)\|\leq C_2\|x_*-x_{N,0}^i\|+C_1tC_3W(m_0,\delta_{\mathbf{x}}^N)+\frac{C_1C_4t}{N}.
\end{equation}

As above we have that
\begin{multline}\label{z_zeta_dev}\|\hat{w}(T,0,x_*,\beta)-\tilde{\zeta}_{N|j}^j(T,\beta;\varrho_N^j)\|
\\\leq \int_{0}^T (L_{g,x}\|\hat{x}(t,0,x_*,\beta)-\tilde{\xi}_N^i(t,\beta;\varrho_N^j)\| +L_{g,m}W(\hat{\mu}[t],\tilde{\nu}_{N|j}(t,\cdot;\varrho_N^j)))dt\\ \leq C_6\|x_*-x_{N,0}^j\|+C_7W(m_0,\delta_{\mathbf{x}}^N)+(L_{g,m}C_4T+C_1C_4T^2L_{g,x})\frac{1}{N}. \end{multline} Denote $C_8=L_{g,m}C_4T+C_1C_4T^2L_{g,x}$.

Hence, we have the estimate
\begin{multline}\label{inequality_sigma_sigma}
  |\sigma(\hat{x}(T,0,x_*,\beta),\hat{\mu}[T])+\hat{w}(T,0,x_*,\beta)- \sigma(\tilde{\xi}_N^j(T,\beta;\varrho_N^j),\tilde{\nu}_N(T,\cdot;\varrho_N^j))- \tilde{\zeta}_N^j(T,\beta;\varrho_N^j)|\\ \leq
  (L_{\sigma,x}C_2+C_6)\|x_*-x_{N,0}^j\|+(L_{\sigma,x}C_5+L_{\sigma,m}C_3+C_7)W(m_0,\delta_{\mathbf{x}}^N)\\+ (L_{\sigma,x}C_1C_4T+L_{\sigma,m}C_4+ C_8)\frac{1}{N}
\end{multline}

It follows from \cite[Theorem 10.4.6]{Bogachev} that for any measure $\varrho_N^j\in\mathcal{P}(\mathcal{U})$ there exists a system of measures $\rho_{x_*}\in\mathcal{P}(\mathcal{U})$ such that for any function $\varphi\in C(\mathcal{U})$
$$\int_{\mathcal{U}}\varphi(\beta)\varrho_N^j(\beta)=N\int_{G_0}m_N^j(dx_{*})\int_{\mathcal{U}}\varphi(\beta)\rho_{x_*}(d\beta). $$
From this (\ref{inequality_rho_rho}), (\ref{inequality_sigma_sigma}) and substitution of the value of $\bar{J}_N^j(\hat{\varrho}_N|\varrho_N^j)$, we get the estimate
\begin{multline*}
\left|\bar{J}_N^j(\hat{\varrho}_N|\varrho_N^j)-N \int_{G_0}m_N^j(dx_*)\int_{\mathcal{U}} (\sigma(\hat{x}(T,0,x_*,\beta),\hat{\mu}[T])+\hat{w}(T,0,x_*,\beta))\rho_{x_*}(d\beta)\right|\\\leq
(L_{\sigma,x}C_2+C_6)N\int_{G_0}\|x_*-x_{N,0}^j\|m_N^j(dx_*) +(L_{\sigma,x}C_5+L_{\sigma,m}C_3+C_7)W(m_0,\delta_{\mathbf{x}}^N)\\ +(L_{\sigma,x}C_1C_4T+L_{\sigma,m}C_4+ C_8)\frac{1}{N}
\end{multline*}

This,  (\ref{main_estimate_nash}), and (\ref{inequality_rho_rho}) yield that
\begin{multline*}
  \bar{J}_N^j(\hat{\varrho}_N)\geq \bar{J}_N^j(\hat{\varrho}_N|\varrho_N^j) -2(L_{\sigma,x}C_2+C_6)N\int_{G_0}\|x_*-x_{N,0}^j\|m_N^j(dx_*) \\ -2(L_{\sigma,x}C_5+L_{\sigma,m}C_3+C_7)W(m_0,\delta_{\mathbf{x}}^N) -(L_{\sigma,x}C_1C_4T+L_{\sigma,m}C_4+ C_8)\frac{1}{N}.
\end{multline*}
This inequality is the conclusion of the Theorem for $\hat{C}_1=2C_8$, $\hat{C_2}=2(L_{\sigma,x}C_2+C_6)$ and $\hat{C_3}=C_9$.
\end{proof}

\end{document}